\documentclass[12pt,reqno]{amsart}
\usepackage{amsmath, amsthm, amssymb, stmaryrd}

\advance \topmargin by -\headheight
\advance \topmargin by -\headsep
     
\evensidemargin \oddsidemargin
\newtheorem{theorem}{Theorem}[section]
\newtheorem{lemma}[theorem]{Lemma}
\newtheorem{corollary}[theorem]{Corollary}

\theoremstyle{definition}

\theoremstyle{remark}

\numberwithin{equation}{section}

\newcommand{\mmod}[1]{\,\,(\text{\rm mod}\,\,#1)}

\def\bfx{{\mathbf x}}
\def\bfy{{\mathbf y}}
\def\bfz{{\mathbf z}}

\def\calC{{\mathcal C}}

\def\what{\widehat w}

\def\dbN{{\mathbb N}}  
\def\dbR{{\mathbb R}}
\def\dbZ{{\mathbb Z}}

\def\gra{{\mathfrak a}}
\def\grB{{\mathfrak B}}

\def\grD{{\mathfrak D}}\def\grE{{\mathfrak E}}

\def\grJ{{\mathfrak J}}

\def\grm{{\mathfrak m}}\def\grM{{\mathfrak M}}
\def\grN{{\mathfrak N}}

\def\grS{{\mathfrak S}}
\def\grB{{\mathfrak B}}

\def\grU{{\mathfrak U}}

\def\alp{{\alpha}} \def\bfalp{{\boldsymbol \alpha}}
\def\bet{{\beta}}  \def\bfbet{{\boldsymbol \beta}}
\def\gam{{\gamma}} 

\def\del{{\delta}}  
\def\zet{{\zeta}}  
\def\bfeta{{\boldsymbol \eta}}

 \def\Lam{{\Lambda}} 

\def\bfxi{{\boldsymbol \xi}}

\def\Ups{{\Upsilon}} 
 
 \def\chitil{{\widetilde \chi}}

\def\eps{\varepsilon}

\def\le{\leqslant} \def\ge{\geqslant}

\def\d{{\,{\rm d}}}

\begin{document}
\title[the major and minor arc integrals]{An instance where the 
major and minor\\ arc integrals meet}
\author[J\"org Br\"udern]{J\"org Br\"udern}
\address{Mathematisches Institut, Bunsenstrasse 3--5, D-37073 G\"ottingen, Germany}
\email{bruedern@uni-math.gwdg.de}
\author[Trevor D. Wooley]{Trevor D. Wooley}
\address{School of Mathematics, University of Bristol, University Walk, Clifton, Bristol BS8 
1TW, United Kingdom}
\email{matdw@bristol.ac.uk}
\date{}
\dedicatory{Dedicated to the memory of Christopher Hooley}
\begin{abstract}We apply the circle method to obtain an asymptotic formula for the 
number of integral points on a certain sliced cubic hypersurface related to the Segre cubic. 
Unusually, the major and minor arc integrals in this application are both positive and of the 
same order of magnitude.\end{abstract}
\maketitle

\section{Introduction} In this communication we concern ourselves with the simultaneous 
Diophantine equations
\begin{equation}\label{1.1}
\sum_{i=1}^{10}x_i^3=\sum_{i=1}^6x_i=\sum_{i=7}^{10}x_i=0,
\end{equation}
and specifically with the number $N(B)$ of integral solutions $\bfx$ of this system 
satisfying $|x_i|\le B$. Associated with (\ref{1.1}) are the $p$-adic densities
\begin{equation}\label{1.2}
\chi_p=\lim_{h\rightarrow \infty} (p^h)^{-7}M_p(h),
\end{equation}
in which, for each prime number $p$, the expression $M_p(h)$ denotes the number of 
solutions of (\ref{1.1}) with $\bfx\in (\dbZ/p^h\dbZ)^s$. Also, one has the real density
\begin{equation}\label{1.3}
\chi_\infty =\lim_{\eta\rightarrow 0+}\, (2\eta)^{-3}M_\infty(\eta),
\end{equation}
where $M_\infty(\eta)$ denotes the volume of the subset of 
$\left[-{\textstyle{\frac{1}{2}}},{\textstyle{\frac{1}{2}}}\right]^{10}$ defined by the 
simultaneous inequalities
\begin{equation}\label{1.4}
\biggl| \sum_{i=1}^{10}x_i^3\biggr|<\eta,\quad \biggl| \sum_{i=1}^6x_i\biggr|<\eta ,
\quad \biggl| \sum_{i=7}^{10}x_i\biggr|<\eta .
\end{equation}
We note that our definition here of the real density employs a unit hypercube, in contrast 
to the commonly favoured box $[-1,1]^{10}$.\par

Our principal goal is an asymptotic formula for $N(B)$.

\begin{theorem}\label{theorem1.1}
One has
\begin{equation}\label{1.5}
N(B)=(45+\calC)(2B)^5+O(B^{5-1/200}),
\end{equation}
where $\calC>0$ is defined by the absolutely convergent product $\calC=\chi_\infty 
\prod_p\chi_p$.
\end{theorem}

The term $\calC(2B)^5$ in (\ref{1.5}) arises from the product of local densities in the 
counting problem under consideration, and is asymptotically equal to the major arc 
integral in a Hardy-Littlewood (circle) method treatment of $N(B)$. The additional term 
$45(2B)^5$ stems from $45$ linear spaces of affine dimension $5$ contained in the variety 
defined by (\ref{1.1}), as exemplified by that defined via the equations 
$x_{2i-1}+x_{2i}=0$ $(1\le i\le 5)$. A notable feature of our treatment of $N(B)$ is that 
this second term may be seen as arising directly from the minor arc integral in an 
application of the circle method.\par

Manin and his collaborators \cite{BM1990, FMT1989} and later Peyre \cite{Pey1995} made 
far-reaching predictions concerning the number of rational points of bounded height on a 
large class of projective varieties. When interpreted in terms of the circle method, one is 
led to the general philosophy, made explicit in the appendix to \cite{VW1995}, that the 
major arc integral accounts for the contribution of the generic points in the solution set, this 
being approximated by the product of local densities. Meanwhile, the minor arc integral 
should approximate the number of points lying on certain special subvarieties. Examples 
involving problems of degree exceeding two in which this heuristic has been substantiated 
are rare. Moreover, hitherto, the derivation of the associated asymptotic formulae did not 
proceed via the circle method, and the alignment with the conjectured behaviour was 
identified only {\it a posteriori} by a formal computation of the major arc integral and 
comparison with the product of local densities. What we have in mind here are such 
approaches as that involving the application of torsors (see for example \cite{BBS2014, 
dlB2007, Pey2001}). In contrast, our theorem is obtained by a direct application of the 
circle method. The special subvarieties in question are the aforementioned $45$ linear 
spaces, each containing approximately $(2B)^5$ integral points in $[-B,B]^{10}$. As our 
analysis shows, the minor arc integral is also approximately $45(2B)^5$. The contributions 
arising from the major and minor arc integrals are consequently of the same order of 
magnitude. Hooley \cite[Chapter II.1]{Hoo1986} has offered evidence in support of the 
conjecture that a similar phenomenon holds also in a circle method approach to counting 
integral solutions of the equation
$$x_1^3+x_2^3+x_3^3=x_4^3+x_5^3+x_6^3.$$

A consequence of Theorem \ref{theorem1.1} is an essentially optimal mean value estimate.

\begin{corollary}\label{corollary1.2} One has
$$\int_{[0,1)^3} \Biggl| \sum_{1\le x,y\le X}e(\alp_1x+\alp_2y+\alp_3(x^3+y^3))
\Biggr|^5\d\bfalp \asymp X^5.$$
\end{corollary}

This conclusion is related to recent work of Bourgain and Demeter \cite{BD2015}. Given 
complex numbers $\gra_{x,y}$, their work establishes the upper bound
$$\int_{[0,1)^3}\biggl| \sum_{1\le x,y\le X}\gra_{x,y}e(\alp_1x+\alp_2y+
\alp_3(x^2+y^2))\biggr|^4\d\bfalp \ll X^\eps \biggl( \sum_{1\le x,y\le X}|\gra_{x,y}|^2
\biggr)^2.$$
With care, one would extract the same conclusion when the binary form $x^2+y^2$ is 
replaced by $x^3+y^3$, and this would deliver the upper bound
$$\int_{[0,1)^3} \biggl| \sum_{1\le x,y\le X}e(\alp_1x+\alp_2y+\alp_3(x^3+y^3))
\biggr|^4\d\bfalp \ll X^{4+\eps}.$$
The conclusion of Corollary \ref{corollary1.2} implies the more general upper bound
\begin{equation}\label{1.6}
\int_{[0,1)^3} \biggl| \sum_{1\le x,y\le X}e(\alp_1x+\alp_2y+\alp_3(x^3+y^3))
\biggr|^s\d\bfalp \asymp X^s+X^{2s-5}
\end{equation}
for all non-negative exponents $s$. Indeed, when $s>5$, little additional effort is required 
to obtain an asymptotic formula for this mean value. It would be desirable to obtain an 
analogue of the upper bound implicit in (\ref{1.6}) in which general complex weights are 
present, thereby providing a strengthening of the conclusions of Bourgain and Demeter 
\cite{BD2015} in this context.\par

After a brief detour in \S2 in which we establish the proof of Corollary \ref{corollary1.2} 
assuming the conclusion of Theorem \ref{theorem1.1}, we attend to the main focus of this 
paper, namely the confirmation of the asymptotic formula for $N(B)$ presented in the latter 
theorem. This is achieved by means of the Hardy-Littlewood method in \S\S3-9. We begin 
in \S3 by describing the infrastructure for our slightly non-standard application of the circle 
method. The minor arc integral is handled in \S4 by utilising what, elsewhere, we have 
referred to as a complification process. Preliminary work on the major arc integral in \S5 is 
followed in \S\S6 and 7 by the completion of the truncated singular integral and singular 
series, respectively. We evaluate these completions in \S\S8 and 9, thereby completing the 
proof of Theorem \ref{theorem1.1}. The evaluation of the completed singular integral in 
\S9 is pursued in some detail in order to make transparent the connection between the 
formal singular integral and the quantity $\chi_\infty$ defined via the Siegel volume in 
(\ref{1.3}). Our treatment should be of sufficient independent interest that scholars may 
find it to be of utility in quite general analyses employing the circle method.\par

Our basic parameter is $B$, a sufficiently large positive number. Whenever $\eps$ appears 
in a statement, either implicitly or explicitly, we assert that the statement holds for each 
$\eps>0$. In this paper, implicit constants in Vinogradov's notation $\ll$ and $\gg$ may 
depend on $\eps$. We write $X\asymp Y$ when $X\ll Y\ll X$. We make frequent use of 
vector notation in the form $\bfx=(x_1,\ldots,x_r)$. Here, the dimension $r$ depends on 
the course of the argument. As is conventional in analytic number theory, we write 
$e(z)$ for $e^{2\pi iz}$, and, when $q\in \dbN$, we put $e_q(z)=e^{2\pi iz/q}$.\par

\noindent {\bf Acknowledgements:} The authors acknowledge support by Akademie der 
Wissenschaften zu G\"ottingen and Deutsche Forschungsgemeinschaft. The second 
author's work was supported by a European Research Council Advanced Grant under the 
European Union's Horizon 2020 research and innovation programme via grant agreement 
No.~695223. 

\section{Deduction of the corollary} We first establish the lower bound implicit in the 
conclusion of Corollary \ref{corollary1.2}. Let
$$g(\alp,\bet)=\sum_{1\le x\le X}e(\alp x^3+\bet x)$$
and write $G(\bfalp)=g(\alp_1,\alp_2)g(\alp_1,\alp_3)$, so that
$$G(\bfalp)=\sum_{1\le x,y\le X}e(\alp_1(x^3+y^3)+\alp_2x+\alp_3y).$$
Also, put $\tau=1/24$, and define the box
$$\grB=[0,\tau X^{-3}]\times [0,\tau X^{-1}]\times [0,\tau X^{-1}].$$
Then whenever $1\le x,y\le X$ and $\bfalp\in \grB$, one has
$$0\le \alp_1(x^3+y^3)+\alp_2x+\alp_3y\le 4\tau,$$
whence
$$|G(\bfalp)| \ge \lfloor X\rfloor^2\cos (8\pi \tau)=
{\textstyle{\frac{1}{2}}}\lfloor X\rfloor^2.$$
Thus,
$$\int_\grB |G(\bfalp)|^5\d\bfalp \gg X^{10}\text{mes}(\grB)\asymp X^5,
$$
and the desired lower bound follows.\par

In order to derive the corresponding upper bound from Theorem \ref{theorem1.1}, we 
begin by noting that
$$|G(\bfalp)|^5\le |g(\alp_1,\alp_2)^4g(\alp_1,\alp_3)^6|+
|g(\alp_1,\alp_2)^6g(\alp_1,\alp_3)^4|.$$
By integrating over the unit cube and applying orthogonality, we thus obtain
$$\int_{[0,1)^3}|G(\bfalp)|^5\d\bfalp \le 2N(X),$$
and the desired upper bound follows from Theorem \ref{theorem1.1}. Having established 
Corollary \ref{corollary1.2}, the inquisitive reader will find the proof of (\ref{1.6}) to be 
routine via suitable applications of H\"older's inequality in combination with the trivial 
estimate $|G(\bfalp)|\le X^2$.

\section{Preliminary manoeuvres} We now embark on our main mission of establishing 
Theorem \ref{theorem1.1}. We begin by introducing the exponential sum
$$f(\alp_1,\alp_2)=\sum_{|x|\le B}e(\alp_1x^3+\alp_2x).$$
Notice that by a change of variables, one has
$${\overline{f(\alp_1,\alp_2)}}=f(-\alp_1,-\alp_2)=f(\alp_1,\alp_2),$$
and hence $f(\alp_1,\alp_2)$ is real. In particular, one has
\begin{equation}\label{3.1}
|f(\alp_1,\alp_2)|^2=f(\alp_1,\alp_2)f(-\alp_1,-\alp_2)=f(\alp_1,\alp_2)^2.
\end{equation}
We shall feel free in what follows to insert or remove absolute values around even powers 
of such generating functions, implicitly making use of the relations (\ref{3.1}), without 
further comment. Then, by orthogonality, one has
\begin{equation}\label{3.2}
N(B)=\int_{[0,1)^3}f(\alp,\bet)^6f(\alp,\gam)^4\d\bfalp ,
\end{equation}
where we use $\bfalp$ to denote $(\alp,\bet,\gam)$.\par

We analyse $N(B)$ by means of the circle method, and this entails a certain 
Hardy-Littlewood dissection of non-standard type. Put $\del=1/9$, and let $\grM$ denote 
the union of the intervals
$$\grM(q,a)=\{ \alp\in [0,1): |\alp-a/q|\le B^{\del-3}\},$$
with $0\le a\le q\le B^\del$ and $(a,q)=1$. Complementary to this set of major arcs are 
the minor arcs $\grm=[0,1)\setminus \grM$. Thus, on writing
\begin{equation}\label{3.3}
N(B;\grB)=\int_\grB \int_0^1\int_0^1 f(\alp,\bet)^6f(\alp,\gam)^4\d\gam \d\bet \d\alp ,
\end{equation}
it follows that
\begin{equation}\label{3.4}
N(B)=N(B;\grM)+N(B;\grm).
\end{equation}

\par Next we introduce the auxiliary integral
\begin{equation}\label{3.5}
u(n;\grB)=\int_\grB \int_0^1 f(\alp,\bet)^6e(-n\alp)\d\bet \d\alp .
\end{equation}
Write also $v(n)$ for the number of representations of the integer $n$ in the form
$$n=x_1^3+x_2^3+x_3^3+x_4^3,$$
with
\begin{equation}\label{3.6}
-B\le x_1,\ldots ,x_4\le B\quad \text{and}\quad x_1+x_2+x_3+x_4=0.
\end{equation}
Thus, we have
$$\sum_{|n|\le 4B^3}u(n;\grB)v(n)=\sum_{x_1,\ldots ,x_4}\int_\grB\int_0^1 
f(\alp,\bet)^6e\left(-(x_1^3+x_2^3+x_3^3+x_4^3)\alp\right)\d\bet \d\alp ,$$
in which the summation over $x_1,\ldots ,x_4$ is again subject to the conditions 
(\ref{3.6}). Hence, by employing orthogonality and recalling the notation (\ref{3.3}), one 
infers that
\begin{equation}\label{3.7}
N(B;\grB)=\sum_{|n|\le 4B^3}u(n;\grB)v(n).
\end{equation}
We compute $N(B;\grm)$ in the next section, and $N(B;\grM)$ in \S\S5--9.

\section{The minor arcs} The point of departure in our analysis of the minor arcs is the 
relation (\ref{3.7}), and here we isolate the term with $n=0$ for special attention. Note 
that $v(0)$ counts the number of integral solutions of the system
\begin{align*}
x_1^3+x_2^3+x_3^3+x_4^3&=0\\
x_1+x_2+x_3+x_4&=0,
\end{align*}
with $|x_i|\le B$ $(1\le i\le 4)$. There are $(2B)^2+O(B)$ integral solutions with 
$x_1+x_2=x_3+x_4=0$, and when $x_1+x_2$ is non-zero, one sees that
$$x_1^2-x_1x_2+x_2^2=x_3^2-x_3x_4+x_4^2,$$
whence $x_1x_2=x_3x_4$. In this second class of solutions, one therefore has 
$\{x_1,x_2\}=\{-x_3,-x_4\}$, and there are $2(2B)^2+O(B)$ such integral solutions. We 
may thus conclude that
\begin{equation}\label{4.1}
v(0)=3(2B)^2+O(B).
\end{equation}

\par Meanwhile, by orthogonality $u(0;[0,1))$ counts the number of integral solutions of 
the system
\begin{equation}\label{4.2}
\sum_{i=1}^6x_i^3=\sum_{i=1}^6x_i=0,
\end{equation}
with $|x_i|\le B$ $(1\le i\le 6)$. By the methods of \cite{dlB2007} and \cite{VW1995}, one 
therefore has
\begin{equation}\label{4.3}
u(0;[0,1)) =15(2B)^3+O\left(B^2(\log B)^5\right) .
\end{equation}
A few words of explanation are required to justify this assertion. What these methods show 
is that there are $O\left(B^2(\log B)^5\right)$ solutions of (\ref{4.2}) with $|x_i|\le B$ not 
lying on the linear spaces defined by the equations
$$x_{j_1}+x_{j_2}=x_{j_3}+x_{j_4}=x_{j_5}+x_{j_6}=0,$$
wherein $\{j_1,j_2,\ldots ,j_6\}=\{1,2,\ldots ,6\}$. The number of such linear spaces is
$$\frac{1}{3!}\binom{6}{2}\binom{4}{2}=15,$$
and each contains $(2B)^3+O(B^2)$ integral points. This confirms the asymptotic formula 
(\ref{4.3}).\par

An asymptotic formula for $u(n;\grm)$ is obtained from a crude upper bound for 
$u(n;\grM)$, which we now derive. Observe that $\text{mes}(\grM)=O(B^{3\del-3})$, 
and by orthogonality one has
\begin{align*}
u(0;\grM)&=\int_\grM \sum_{\substack{x_1+\ldots +x_6=0\\ -B\le x_1,\ldots ,x_6\le B}}
e\left( \alp (x_1^3+\ldots +x_6^3)\right) \d\alp \\
&\ll B^5\,\text{mes}(\grM)\ll B^{2+3\del}.
\end{align*}  
Consequently,
$$u(0;\grm)=u(0;[0,1))-u(0;\grM)=15(2B)^3+O(B^{2+3\del}),$$
and it follows from (\ref{4.1}) that
\begin{equation}\label{4.4}
u(0;\grm)v(0)=45(2B)^5+O(B^{4+3\del}).
\end{equation}

\par When $n$ is non-zero, it follows from its definition that $v(n)$ is bounded above by 
the number of integral solutions of the equation
$$n=(x_1^3+x_2^3+x_3^3)-(x_1+x_2+x_3)^3,$$
and hence of
$$n=-3(x_1+x_2)(x_2+x_3)(x_3+x_1).$$
By employing an elementary estimate for the divisor function, one sees that for a fixed 
non-zero integer $n$, there are $O(n^\eps)$ possible choices for $x_1+x_2$, $x_2+x_3$ 
and $x_3+x_1$, and hence also for $x_1$, $x_2$ and $x_3$. Thus $v(n)=O(n^\eps)$. It 
therefore follows from (\ref{3.5}) via the inequalities of Cauchy and Bessel that
\begin{align*}
\Biggl( \sum_{0<|n|\le 4B^3}u(n;\grm)v(n)\Biggr)^2&\ll B^{3+\eps}\sum_{|n|\le 4B^3}
|u(n;\grm)|^2\\
&\ll B^{3+\eps} \int_\grm \biggl( \int_0^1 f(\alp,\bet)^6\d\bet \biggr)^2\d\alp \\
&\ll B^{3+\eps} \int_\grm \int_0^1\int_0^1 f(\alp,\bet)^6f(\alp,\gam)^6 
\d\gam\d\bet\d\alp .
\end{align*}

\par As a consequence of Weyl's inequality (see \cite[Lemma 2.5]{Vau1997}), one sees 
that
$$\sup_{\alp\in\grm}\sup_{\gam\in \dbR} |f(\alp,\gam)|\ll B^{1-\del/4+\eps}.$$
Thus, by reference to (\ref{3.2}), we deduce that
\begin{align*}
\Biggl( \sum_{0<|n|\le 4B^3}u(n;\grm)v(n)\Biggr)^2&\ll B^{3+\eps}(B^{1-\del/4})^2
\int_{[0,1)^3}f(\alp,\bet)^6f(\alp,\gam)^4\d\bfalp \\
&\ll B^{5-\del/3}N(B),
\end{align*}
whence
$$\sum_{0<|n|\le 4B^3}u(n;\grm)v(n)\ll B^{5/2-\del/6}N(B)^{1/2}.$$
On recalling (\ref{3.7}) and (\ref{4.4}), we conclude that
\begin{equation}\label{4.5}
N(B;\grm)=45(2B)^5+O\left( B^{4+3\del}+B^{5/2-\del/6}N(B)^{1/2}\right).
\end{equation}

\section{The major arcs} We begin by extending our definition of one dimensional major 
arcs to corresponding three dimensional arcs. When $0\le a,b,c\le q\le B^\del$ and 
$(a,q)=1$, let $\grN(q,a,b,c)$ denote the set of triples $(\alp,\bet,\gam)\in [0,1)^3$ such 
that $\alp\in \grM(q,a)$,
$$-\frac{1}{2q}\le \bet -\frac{b}{q}<\frac{1}{2q}\quad \text{and}\quad -\frac{1}{2q}\le 
\gam -\frac{c}{q}<\frac{1}{2q}.$$
We observe that these boxes $\grN(q,a,b,c)$ are disjoint, and hence their union $\grN$ is 
equal to $\grM\times [0,1)^2$.\par

We next introduce the standard major arc approximants to $f(\alp_1,\alp_2)$. When 
$a_1,a_2\in \dbZ$ and $q\in \dbN$, we put
$$S(q,a_1,a_2)=\sum_{r=1}^qe_q(a_1r^3+a_2r).$$
Likewise, when $\bet_1,\bet_2\in \dbR$, we write
$$v(\bet_1,\bet_2;B)=\int_{-B}^B e(\bet_1\gam^3+\bet_2\gam)\d\gam .$$
It is convenient in what follows to abbreviate $v(\bet_1,\bet_2;B)$ to $v(\bet_1,\bet_2)$, 
and $v(\bet_1,\bet_2;\tfrac{1}{2})$ to $v_1(\bet_1,\bet_2)$. Finally, when 
$(\alp,\bet,\gam)\in \grN(q,a,b,c)\subseteq \grN$, we write
\begin{equation}\label{5.a1}
f^*(\alp,\bet)=q^{-1}S(q,a,b)v(\alp-a/q,\bet-b/q)
\end{equation}
and
\begin{equation}\label{5.a2} 
f^*(\alp,\gam)=q^{-1}S(q,a,c)v(\alp-a/q,\gam-c/q).
\end{equation}
Notice that the cosmetic ambiguity in this definition would arise only when $\bet=\gam$, in 
which case the two definitions coincide. Then, by appealing to the final conclusion 
(2.4) of \cite[Theorem 3]{BR2015}, we see that when $(\alp,\bet,\gam)\in \grN(q,a,b,c)
\subseteq \grN$, one has
\begin{equation}\label{5.1}
f(\alp,\bet)-f^*(\alp,\bet)\ll q^{2/3+\eps}\le B^\del
\end{equation}
and
\begin{equation}\label{5.2}
f(\alp,\gam)-f^*(\alp,\gam)\ll q^{2/3+\eps}\le B^\del .
\end{equation}

\par Our first lemma relates $N(B;\grM)$ to the mean value
\begin{equation}\label{5.c}
N^*(B)=\int_\grN f^*(\alp,\bet)^6f^*(\alp,\gam)^4\d\bfalp .
\end{equation}

\begin{lemma}\label{lemma5.1} One has $N(B;\grM)-N^*(B)\ll B^{5-\del}$.
\end{lemma}

\begin{proof} Suppose that $(\alp,\bet,\gam)\in \grN$. Then, by making use of the trivial 
estimate $f(\alp_1,\alp_2)=O(B)$ in combination with (\ref{5.1}) and (\ref{5.2}), we see 
that
$$f^*(\alp,\bet)^6f^*(\alp,\gam)^4=\left( f(\alp,\bet)+O(B^\del)\right)^6\left( 
f(\alp,\gam)+O(B^\del)\right)^4,$$
whence
\begin{equation}\label{5.3}
f^*(\alp,\bet)^6f^*(\alp,\gam)^4-f(\alp,\bet)^6f(\alp,\gam)^4\ll B^\del T_1(\bfalp)
+B^{2\del }T_2(\bfalp)+B^{7+3\del},
\end{equation}
where
$$T_1(\bfalp)=|f(\alp,\bet)^5f(\alp,\gam)^4|+|f(\alp,\bet)^6f(\alp,\gam)^3|$$
and
$$T_2(\bfalp)=|f(\alp,\bet)^4f(\alp,\gam)^4|+|f(\alp,\bet)^5f(\alp,\gam)^3|+
|f(\alp,\bet)^6f(\alp,\gam)^2|.$$
Making use of the trivial estimate $|z_1\cdots z_n|\le |z_1|^n+\ldots +|z_n|^n$, we find 
that
$$T_1(\bfalp)\ll |f(\alp,\bet)^7f(\alp,\gam)^2|+|f(\alp,\bet)^2f(\alp,\gam)^7|$$
and
$$T_2(\bfalp)\ll f(\alp,\bet)^6f(\alp,\gam)^2+f(\alp,\bet)^2f(\alp,\gam)^6.$$

\par Integrating $T_1(\bfalp)$ over $\bfalp\in \grN$ and invoking symmetry, we obtain
$$\int_\grN T_1(\bfalp)\d\bfalp \ll \int_0^1 \int_0^1 |f(\alp,\bet)|^7\int_0^1 
|f(\alp,\gam)|^2\d\gam \d\bet \d\alp .$$
By orthogonality, therefore, together with an application of H\"older's inequality, we see that
\begin{align*}
\int_\grN T_1(\bfalp)\d\bfalp &\ll B\int_0^1\int_0^1 |f(\alp,\bet)|^7\d\bet \d\alp\\
&\ll B\Ups_6^{3/4}\Ups_{10}^{1/4},
\end{align*}
where for even exponents $t$ we write
$$\Ups_t=\int_0^1\int_0^1 f(\alp,\bet)^t\d\bet \d\alp .$$
The case $k=3$ of \cite[Lemma 5.2]{Hua1965} delivers the bound
\begin{equation}\label{5.4}
\Ups_6\ll B^{3+\eps},
\end{equation}
and by equation (2) in \cite[\S5 of Chapter V]{Hua1965}, meanwhile, one has 
$\Ups_{10}\ll B^{6+\eps}$. Note that these sources in fact count solutions of the 
underlying Diophantine equations in which the variables are positive, whereas our bounds 
assert that the number of solutions, both positive and negative, be so bounded. The reader 
should have no difficulty, however, in either adapting the methods underlying these cited 
bounds, or indeed deriving the stated results through application of the triangle inequality. 
Improved bounds for the former mean value are the subject 
of \cite{VW1995}, whilst the second is handled more precisely in \cite{BR2015} and 
\cite{Woo2015}. 
Thus we obtain the estimate
$$\int_\grN T_1(\bfalp)\d\bfalp \ll B^{1+\eps}(B^3)^{3/4}(B^6)^{1/4}\ll B^{5-2\del}.$$

\par Similarly, in view of (\ref{5.4}), one deduces that
\begin{align*}
\int_\grN T_2(\bfalp)\d\bfalp &\ll \int_0^1\int_0^1 f(\alp,\bet)^6\int_0^1 f(\alp,\gam)^2
\d\gam \d\bet \d\alp \\
&\ll B \int_0^1\int_0^1 f(\alp,\bet)^6\d\bet \d\alp \ll B^{4+\eps}.
\end{align*}
Since, in addition, one has $\text{mes}(\grM)\ll B^{2\del-3}$, we conclude from (\ref{5.c}) 
and (\ref{5.3}) that
\begin{align*}
N^*(B)-N(B;\grM)&\ll B^\del (B^{5-2\del})+B^{2\del}(B^{4+\eps})+B^{7+3\del}
(B^{3\del-3})\\
&\ll B^{5-\del}.
\end{align*}
This completes the proof of the lemma.
\end{proof}

On combining (\ref{4.5}) and the conclusion of Lemma \ref{lemma5.1} by means of 
(\ref{3.4}), we may conclude thus far that
\begin{equation}\label{5.5}
N(B)-N^*(B)-45(2B)^5\ll B^{5/2-\del/6}N(B)^{1/2}+B^{5-\del}.
\end{equation}
The remainder of this paper will be consumed by the task of showing that
$$N^*(B)=\calC (2B)^5+O(B^{5-\del/20}).$$
As is evident from (\ref{5.5}), this asymptotic relation suffices to confirm that
$$N(B)-(45+\calC)(2B)^5\ll B^{5/2-\del/6}N(B)^{1/2}+B^{5-\del/20},$$
whence
\begin{equation}\label{5.5a}
N(B)=(45+\calC)(2B)^5+O(B^{5-\del/20}).
\end{equation}
This confirms (\ref{1.5}) and, subject to verifying that $\calC>0$, completes the proof of 
Theorem \ref{theorem1.1}.\par

Before proceeding further, we introduce the truncated singular integral
\begin{equation}\label{5.6}
I(q)=\int_{-B^{\del-3}}^{B^{\del-3}}\int_{-1/(2q)}^{1/(2q)}\int_{-1/(2q)}^{1/(2q)} 
v(\xi,\eta)^6v(\xi,\zet)^4\d\zet \d\eta \d\xi
\end{equation}
and the auxiliary sum
\begin{equation}\label{5.7}
A(q)=\sum^q_{\substack{a=1\\ (a,q)=1}}\sum_{b=1}^q
\sum_{c=1}^q q^{-10}S(q,a,b)^6S(q,a,c)^4.
\end{equation}
Then in view of the definitions (\ref{5.a1}), (\ref{5.a2}) and (\ref{5.c}), we may write
\begin{equation}\label{5.8}
N^*(B)=\sum_{1\le q\le B^\del} I(q)A(q).
\end{equation}

\section{The completion of the singular integral} Our next step is to complete the truncated 
singular integral defined in (\ref{5.6}) so as to obtain the completed singular integral
\begin{equation}\label{6.1}
\grJ(B)=\int_{\dbR^3}v(\xi,\eta)^6v(\xi,\zet)^4\d\zet \d\eta \d\xi .
\end{equation}
In pursuit of this goal, we recall from \cite[Theorem 7.3]{Vau1997} the bound
$$v(\alp_1,\alp_2)\ll B(1+|\alp_2|B+|\alp_1|B^3)^{-1/3}.$$
This delivers the estimate
$$v(\xi,\eta)^6v(\xi,\zet)^4\ll B^{10}\left((1+|\xi|B^3)(1+|\eta|B)(1+|\zet|B)
\right)^{-10/9}.$$
In particular, it is apparent from (\ref{6.1}) that the singular integral $\grJ(B)$ exists, and 
that
\begin{equation}\label{6.2}
\grJ(B)\ll B^5.
\end{equation}
Also, when $1\le q\le B^\del$, one sees that whenever
$$(\xi,\eta, \zet)\in \dbR^3\setminus \left([-B^{\del-3},B^{\del-3}]\times 
[-1/(2q),1/(2q)]^2\right),$$
then
$$(1+|\xi|B^3)(1+|\eta|B)(1+|\zet|B)>B^\del .$$
Hence,
\begin{align}
\grJ(B)-I(q)&\ll B^{10-\del/18}\int_{\dbR^3}\left( (1+|\xi|B^3)(1+|\eta|B)(1+|\zet|B)
\right)^{-19/18}\d\zeta \d\eta\d\xi \notag \\
&\ll B^{5-\del/18}.\label{6.3}
\end{align}

\par We summarise (\ref{6.2}) and (\ref{6.3}) in the form of a lemma.

\begin{lemma}\label{lemma6.1} When $1\le q\le B^\del$, one has
$$I(q)\ll B^5\quad \text{and}\quad I(q)=\grJ(B)+O(B^{5-\del/18}).$$
\end{lemma}

In preparation for the next section of our argument, we introduce the truncated singular 
series
\begin{equation}\label{6.4}
\grS(Q)=\sum_{1\le q\le Q}A(q),
\end{equation}
in which $A(q)$ is defined by (\ref{5.7}). Then, on substituting the conclusion of Lemma 
\ref{lemma6.1} into (\ref{5.8}), we infer that
\begin{equation}\label{6.5}
N^*(B)=\left( \grJ(B)+O(B^{5-\del/18})\right) \grS(B^\del).
\end{equation}

\section{The completion of the singular series} The next phase of our discussion is focused 
on the completion of the truncated singular series (\ref{6.4}), thereby delivering the 
completed singular series
\begin{equation}\label{7.1}
\grS=\sum_{q=1}^\infty A(q).
\end{equation}
This step in our argument makes use of the following auxiliary lemma.

\begin{lemma}\label{lemma7.1} When $(a,q)=1$, one has
$$\sum_{b=1}^q S(q,a,b)^4\ll q^{3+\eps}.$$
\end{lemma}

\begin{proof} By orthogonality, one has
\begin{align}
q^{-1}\sum_{b=1}^q&\sum_{r_1,\ldots,r_4\mmod{q}}e_q
\left( a(r_1^3+\ldots +r_4^3)+b(r_1+\ldots +r_4)\right) \notag \\
&=\sum_{\substack{r_1,\ldots ,r_4\mmod{q}\\ r_1+\ldots +r_4\equiv 0\mmod{q}}}
e_q\left( a(r_1^3+\ldots +r_4^3)\right) \notag \\
&=\sum_{r_1,r_2,r_3\mmod{q}}e_q\left( -3a(r_1+r_2)(r_2+r_3)(r_3+r_1)\right) .
\label{7.2}
\end{align}

\par On substituting $u=r_1+r_2$ and $v=r_2+r_3$, we see that the right hand side of 
(\ref{7.2}) is equal to
$$\sum_{u,v,r_3\mmod{q}}e_q\left(-3auv(2r_3+u-v)\right) .$$
The sum over $r_3$ here makes a non-zero contribution only when $q|6auv$. Thus, since 
$(a,q)=1$, we find that this expression is bounded above in absolute value by 
$$q\sum_{\substack{u,v\mmod{q}\\ q|6uv}}1\ll q^{2+\eps}.$$
On recalling (\ref{7.2}), the conclusion of the lemma now follows.
\end{proof}

We next recall an estimate due to Hua (see \cite[Theorem 7.1]{Vau1997}) asserting that 
whenever $(a_1,q)=1$, one has
$$S(q,a_1,a_2)\ll q^{2/3+\eps}.$$
On recalling (\ref{5.7}), we deduce via Lemma \ref{lemma7.1} that
\begin{equation}\label{7.3}
A(q)\ll q^{\eps-2/3}\sum^q_{\substack{a=1\\ (a,q)=1}}\biggl( q^{-4}
\sum_{b=1}^qS(q,a,b)^4\biggr)^2\ll q^{3\eps -5/3}.
\end{equation}
It therefore follows from (\ref{6.4}) and (\ref{7.1}) that the singular series 
$\grS=\underset{Q\rightarrow \infty}\lim\grS(Q)$ converges, and moreover that 
$\grS-\grS(B^\del)\ll B^{-\del/2}$. In particular, on recalling Lemma \ref{lemma6.1} and 
(\ref{6.5}), we may conclude thus far that
\begin{equation}\label{7.4}
N^*(B)=\grS \grJ(B)+O(B^{5-\del/18}).
\end{equation}

\section{The evaluation of the singular series} We have yet to interpret the singular series 
as a product of local densities, the first step towards this goal being that of establishing the 
multiplicative nature of $A(q)$. We put
$$\Psi(\bfy)=\sum_{i=1}^8y_i^3-(y_1+\ldots +y_5)^3-(y_6+y_7+y_8)^3,$$
and then set
$$T(q,a)=\sum_{y_1,\ldots ,y_8\mmod{q}}e_q\left(a\Psi(\bfy)\right) .$$
Then, by orthogonality, we may eliminate two variables and infer that
$$q^{-2}\sum_{b=1}^q\sum_{c=1}^qS(q,a,b)^6S(q,a,c)^4=T(q,a).$$
Thus we have
$$\sum^q_{\substack{a=1\\ (a,q)=1}}q^{-8}T(q,a)=A(q).$$
The standard theory of singular series (see 
\cite[Lemmata 2.10 and 2.11]{Vau1997}) therefore shows that $A(q)$ is a multiplicative 
function of $q$.\par

Observe next that when $p$ is prime and $H$ is a non-negative integer, it follows from 
(\ref{7.3}) that
$$\sum_{h=0}^HA(p^h)=1+O(p^{-3/2}).$$
Thus, writing
$$\chitil_p=\sum_{h=0}^\infty A(p^h),$$
we have $\chitil_p=1+O(p^{-3/2})$, and hence the product $\prod_p\chitil_p$ converges. 
But one has
$$p^{7H}\sum_{h=0}^HA(p^h)=p^{-H}\sum_{a=1}^{p^H}T(p^H,a),$$
and by orthogonality, this is equal to the number of solutions of the congruence 
$\Psi(\bfy)\equiv 0\mmod{p^H}$, with $1\le y_i\le p^H$ $(1\le i\le 8)$. Next back 
substituting $y_9=-(y_1+\ldots +y_5)$ and $y_{10}=-(y_6+y_7+y_8)$, it is apparent that 
this, in turn, is equal to $M_p(H)$. Hence
$$\chitil_p=\lim_{H\rightarrow \infty}\sum_{h=0}^HA(p^h)=\lim_{H\rightarrow \infty}
(p^H)^{-7}M_p(H)=\chi_p,$$
where $\chi_p$ is defined as in (\ref{1.2}).\par

Since $A(q)$ is multiplicative and the series $\grS=\sum_{q=1}^\infty A(q)$ is absolutely 
convergent, we deduce that
\begin{equation}\label{8.1}
\grS=\prod_p\sum_{h=0}^\infty A(p^h)=\prod_p\chi_p.
\end{equation}
Here we note that, by the positivity evident in (\ref{5.7}), one has $\chi_p\ge 1$ for each 
prime number $p$, whence $\grS\ge 1$. In particular, the singular series $\grS$ is 
positive.

\section{The evaluation of the singular integral} The evaluation of the singular integral 
$\grJ(B)$ begins with the observation that, by employing a change of variable in 
(\ref{6.1}), one obtains
\begin{equation}\label{9.1}
\grJ(B)=B^5\int_{\dbR^3}v_1(\xi,\eta)^6v_1(\xi,\zet)^4\d\zet\d\eta\d\xi =B^5\grJ(1).
\end{equation}
We next pursue a strategy proposed by Schmidt \cite{Sch1982, Sch1985} in a somewhat 
refined form.\par

When $0<\eta\le 1$, we define the auxiliary function
\begin{equation}\label{9.2}
w_\eta(\beta)=\eta \biggl( \frac{\sin(\pi\eta \bet)}{\pi \eta \beta}\biggr)^2,
\end{equation}
having Fourier transform
\begin{equation}\label{9.3}
\what_\eta(\gam)=\int_{-\infty}^\infty w_\eta(\bet)e(-\bet \gam)\d\bet =
\max \{ 0,1-|\gam|/\eta \}.
\end{equation}
Here, the integral converges absolutely. One can apply the formula (\ref{9.3}) to 
construct a continuous approximation to the indicator function of a box. When 
$0<\del<\eta$, we define
$$W_{\eta,\del}(\gam)=\begin{cases} 1,&\text{when $|\gam|\le \eta$,}\\
1-{\displaystyle{\frac{|\gam|-\eta}{\del}}},&\text{when $\eta<|\gam|<\eta+\del$,}\\
0,&\text{when $|\gam|\ge \eta+\del$.}\end{cases}$$
Next we put
$$W_\eta^+(\gam)=W_{\eta,\eta^2}(\gam)\quad \text{and}\quad W_\eta^-(\gam)=
W_{\eta-\eta^2,\eta^2}(\gam).$$
We observe that $W_\eta^+(\gam)$ and $W_\eta^-(\gam)$ supply upper and lower 
bounds for the characteristic function of the interval $[-\eta,\eta]$.\par

Since it follows from (\ref{9.3}) that
$$W_{\eta,\del}(\gam)=(1+\eta/\del)
\what_{\eta+\del}(\gam)-(\eta/\del)\what_\eta(\gam),$$
we find that
$$W_\eta^+(\gam)=(1+\eta^{-1})\what_{\eta+\eta^2}(\gam)-\eta^{-1}
\what_\eta(\gam)$$
and
$$W_\eta^-(\gam)=\eta^{-1}\what_\eta(\gam)+(1-\eta^{-1})
\what_{\eta-\eta^2}(\gam).$$

\par Next set
$$F_1(\bfxi)=\sum_{i=1}^{10}\xi_i^3,\quad F_2(\bfxi)=\sum_{i=1}^6\xi_i,
\quad F_3(\bfxi)=\sum_{i=7}^{10}\xi_i,$$
and define
\begin{equation}\label{9.3z}
V_\eta^\pm (\bfxi)=\prod_{i=1}^3W_\eta^\pm (F_i(\bfxi)).
\end{equation}
For convenience, we write $\grU$ for the unit box 
$\left[-\tfrac{1}{2},\tfrac{1}{2}\right]^{10}$. Then our discussion thus far demonstrates 
that the volume $M_\infty(\eta)$ defined via (\ref{1.4}) satisfies
\begin{equation}\label{9.4}
\int_\grU V_\eta^-(\bfxi)\d\bfxi \le M_\infty (\eta)\le \int_\grU V_\eta^+(\bfxi)\d\bfxi .
\end{equation}
We now proceed by turning our attention to the Fourier side. Define
\begin{equation}\label{9.4a}
U(\bfeta)=\int_\grU \prod_{i=1}^3 \eta_i^{-1}\what_{\eta_i}(F_i(\bfxi))\d\bfxi .
\end{equation}

\begin{lemma}\label{lemma9.1} Let $\eta$ be a real number with $0<\eta<1$. Suppose 
that $\eta_i$ is a real number with $|\eta_i-\eta|\le \eta^2$ for $i=1,2,3$. Then
$$U(\bfeta)=\grJ(1)+O(\eta^{1/36}).$$
\end{lemma}

\begin{proof} Put
$$K(\bfbet)=\prod_{i=1}^3\eta_i^{-1}w_{\eta_i}(\bet_i).$$
Then by interchanging orders of integration, it follows from (\ref{9.4a}) that
$$U(\bfeta)=\int_{\dbR^3}v_1(\bet_1,\bet_2)^6v_1(\bet_1,\bet_3)^4K(\bfbet)\d\bfbet .$$
Hence, on recalling (\ref{6.1}), we see that
\begin{equation}\label{9.5}
U(\bfeta)-\grJ(1)=\int_{\dbR^3}v_1(\bet_1,\bet_2)^6v_1(\bet_1,\bet_3)^4\left( 
K(\bfbet)-1\right) \d\bfbet .
\end{equation}

\par Next, put $\grD=[-\eta^{-1/2},\eta^{-1/2}]^3$ and $\grE=\dbR^3\setminus \grD$. 
From the power series expansion of $w_\eta(\bet)$ underlying (\ref{9.2}), we have
$$0\le 1-K(\bfbet)\ll \min \{ 1, \eta^2(\bet_1^2+\bet_2^2+\bet_3^2)\}.$$
Hence, the absolute convergence of the integral $\grJ(1)$ ensures that the contribution 
from integrating over $\grD$ on the right hand side of (\ref{9.5}) is at most
$$\sup_{\bfbet \in \grD}|1-K(\bfbet)| \int_{\dbR^3}v_1(\bet_1,\bet_2)^6
v_1(\bet_1,\bet_3)^4\d\bfbet \ll \eta .$$
Meanwhile, on making use of an argument akin to that delivering (\ref{6.3}), one finds that 
the corresponding contribution from $\grE$ is bounded above by
$$ \int_\grE \left( (1+|\beta_1|)(1+|\beta_2|)(1+|\bet_3|)\right)^{-10/9}\d\bfbet 
\ll (\eta^{-1/2})^{-1/18}=\eta^{1/36}.$$
Thus we infer from (\ref{9.5}) that
$$U(\bfeta)-\grJ(1)\ll \eta^{1/36},$$
completing the proof of the lemma.
\end{proof}

By using the definitions of $W_\eta^\pm (\gam)$ to expand the products (\ref{9.3z}) 
defining $V_\eta^\pm (\bfxi)$ as a linear combination of terms of the shape
$$\what_{\eta_1}(F_1(\bfxi))\what_{\eta_2}(F_2(\bfxi)\what_{\eta_3}(F_3(\bfxi)),$$
it follows from Lemma \ref{lemma9.1} that
\begin{align}
\int_\grU V_\eta^+(\bfxi)\d\bfxi &=\left( (\eta+\eta^2)(1+1/\eta)-\eta(1/\eta)\right)^3
\left( \grJ(1)+O(\eta^{1/36})\right) \notag\\
&=\left(8\eta^3+O(\eta^4)\right) \left( \grJ(1)+O(\eta^{1/36})\right) \label{9.6}
\end{align}
and
\begin{align}
\int_\grU V_\eta^-(\bfxi)\d\bfxi &=\left( \eta(1/\eta)+(\eta-\eta^2)(1-1/\eta)\right)^3
\left( \grJ(1)+O(\eta^{1/36})\right) \notag\\
&=\left(8\eta^3+O(\eta^4)\right) \left( \grJ(1)+O(\eta^{1/36})\right) .\label{9.7}
\end{align}

\par We conclude from (\ref{9.4}), (\ref{9.6}) and (\ref{9.7}) that
$$(2\eta)^{-3}M_\infty (\eta)=\grJ(1)+O(\eta^{1/36}).$$
Consequently, the definition (\ref{1.3}) shows that
$$\chi_\infty =\lim_{\eta\rightarrow 0+}\left( \grJ(1)+O(\eta^{1/36})\right) =\grJ(1).$$
On recalling (\ref{7.4}), (\ref{8.1}) and (\ref{9.1}), we may conclude that
$$N^*(B)=B^5\grS\grJ(1)+O(B^{5-\del/18})=\calC B^5+O(B^{5-\del/18}),$$
where $\calC=\chi_\infty \prod_p\chi_p$. The conclusion of Theorem \ref{theorem1.1} now 
follows on verifying that $\calC>0$, as described in the argument leading to (\ref{5.5a}) 
above.\par

It remains to confirm that the real density $\chi_\infty$ is positive. This is routine. Observe 
first that the point
$$\bfx_0=(1/\sqrt{3},0,0,-1/\sqrt{3},0,\ldots, 0)$$
is a non-singular solution of the system of equations
\begin{equation}\label{9.z}
F_1(\bfx)=F_2(\bfx)=F_3(\bfx)=0.
\end{equation}
By the Implicit Function Theorem (see \cite[Theorem 7-6]{Apo1974}), there exists a 
positive number $\del$ with the property that whenever
$$|z_4+1/\sqrt{3}|<\del\quad \text{and}\quad |z_i|<\del\qquad(5\le i\le 10),$$
then the equations (\ref{9.z}) possesses a solution $(x_1,x_2,x_3)$ with coordinates 
$x_i=x_i(\bfz)$ satisfying
$$|x_1(\bfz)-1/\sqrt{3}|\ll \del,\quad |x_2(\bfz)|\ll \del,\quad |x_3(\bfz)|\ll \del .$$
Here, the implicit constants are absolute. Notice that whenever $\del$ is sufficiently small, 
then the partial derivatives of the polynomials $F_i(x_1,x_2,x_3,\bfz)$ $(i=1,2,3)$ remain 
close to their values at $\bfx_0$. It therefore follows that there is an absolute constant 
$\Lam$ having the property that whenever $|\eta_i|<\Lam \eta$ $(i=1,2,3)$, then
$$|F_i(x_1(\bfz)+\eta_1,x_2(\bfz)+\eta_2,x_3(\bfz)+\eta_3,z_4,\ldots ,z_{10})|<\eta.$$
In particular,
\begin{align*}
M_\infty (\eta)&\ge (2\Lam \eta)^3\text{mes} \{ (z_4,\ldots ,z_{10})\in (-\del,\del)^7\}
\\
&=(2\Lam \eta)^3(2\del)^7\gg \eta^3.
\end{align*} 
Hence
$$\chi_\infty =\lim_{\eta \rightarrow 0+}(2\eta)^{-3}M_\infty (\eta)>0.$$

\par Our method to treat the singular integral works in broad generality. Write 
$\grU_s=\left[ -\tfrac{1}{2},\tfrac{1}{2}\right]^s$. Suppose that an affine variety is
defined by polynomial equations
$$F_1(x_1,\ldots ,x_s)=\ldots =F_r(x_1,\ldots ,x_s)=0.$$
Then, whenever the formal singular integral
$$I=\int_{\dbR^r}\int_{}
e\left(\bet_1F_1(\bfx)+\ldots +\bet_rF_r(\bfx)\right) \d\bfx \d\bfbet $$
converges absolutely, the above method shows that
$$I=\lim_{\eta\rightarrow 0+}(2\eta)^{-r}M_\infty (\eta),$$
where now $M_\infty(\eta)$ is the Lebesgue measure of the set of all $\bfx\in \grU_s$ 
for which $|F_j(\bfx)|<\eta$ $(1\le j\le r)$. Further, if the variety contains a non-singular 
point $\bfx_0\in \grU_s$, then our methods also show that $I>0$.

\bibliographystyle{amsbracket}
\providecommand{\bysame}{\leavevmode\hbox to3em{\hrulefill}\thinspace}

\end{document}